\documentclass[11pt]{amsart}
\usepackage{amssymb, amsmath, amsthm}
\usepackage{amsrefs}
\usepackage{graphicx}
\usepackage[final]{hyperref}
\usepackage[a4paper, centering]{geometry}
\usepackage{color}
\usepackage{graphicx} 

\newtheorem{theorem}{Theorem}

\newtheorem{lemma}[theorem]{Lemma}
\newtheorem{corollary}[theorem]{Corollary}
\theoremstyle{definition}
\newtheorem{definition}[theorem]{Definition}
\theoremstyle{remark}
\newtheorem{remark}[theorem]{Remark}

\def\R{\mathbb{R}}

\def\pscal#1#2{\left\langle#1,\,#2\right\rangle}

\def \e{\varepsilon}

\newcommand{\nplap}{\Delta^N_p}

\DeclareMathOperator{\Trace}{Tr}

\begin{document}

\title[Normalized $p$--Laplacian]%
{On the first eigenvalue\\ of the normalized p-Laplacian}%
\author[G.~Crasta, I.~Fragal\`a, B.~Kawohl]{Graziano Crasta,  Ilaria Fragal\`a, Bernd Kawohl}
\address[Graziano Crasta]{Dipartimento di Matematica ``G.\ Castelnuovo'', Univ.\ di Roma I\\
P.le A.\ Moro 2 -- 00185 Roma (Italy)}
\email{crasta@mat.uniroma1.it}

\address[Ilaria Fragal\`a]{
Dipartimento di Matematica, Politecnico\\
Piazza Leonardo da Vinci, 32 --20133 Milano (Italy)
}
\email{ilaria.fragala@polimi.it}

\address[Bernd Kawohl]{Mathematisches Institut, Universit\"at zu K\"oln, 50923 K\"oln (Germany)
}
\email{kawohl@math.uni-koeln.de}



\date{November 25, 2018}

\begin{abstract} We prove that, if $\Omega$ is an open bounded domain with smooth and connected boundary, for every $p \in (1,  + \infty)$ the first Dirichlet eigenvalue of the normalized $p$-Laplacian is simple in the sense that two positive eigenfunctions are necessarily multiple of each other. We also give a (non-optimal) lower bound  for the eigenvalue in terms of the measure of $\Omega$, 
and we address the open problem of proving a Faber-Krahn type inequality with balls as optimal domains. 
\end{abstract}

\maketitle

\medskip
\section{Introduction and statement of the results}

Given an open bounded subset $\Omega$ of $\R^n$, 
we consider the following eigenvalue problem
\begin{equation}\label{f:pb1} \begin{cases}
 - \nplap u = \lambda_p u & \text{ in } \Omega
 \\  \noalign{\medskip}
 u = 0 & \text{ on } \partial \Omega\,,
 \end{cases}
 \end{equation}
where $\nplap$ denotes the normalized or game-theoretic $p$-Laplacian, defined for any $p \in (1, + \infty)$ by 
\[
\begin{array}{ll} 
\nplap u & \displaystyle := \frac{1}{p} |\nabla u| ^ {p-2}  {\rm div} \big (|\nabla u| ^ {p-2} \nabla u \big )  
\\ \noalign{\medskip} 
& \displaystyle 
= \frac{p-2}{p} |\nabla u|^{-2} \pscal{\nabla ^2 u\, \nabla u}{\nabla u} + \frac{1}{p}\, \Trace(\nabla ^2 u)\,,
\end{array}
\]
 where $\nabla^2u$  stands for the Hessian of $u$.
Equivalently, see \cite{K0}, it can be defined as a convex combination of the limit operators 
as $p \to 1$ and $p \to + \infty$, since
\begin{equation}\label{convexcomb}\nplap u =\frac{p-1}{p} \Delta _\infty ^N u +  \frac{1}{p}  \Delta _1 ^N u  \,,
\end{equation} 
with 
$$
\Delta _\infty ^N u = \frac{1}{|\nabla u| ^ 2} \pscal{\nabla ^ 2 u \nabla u} {\nabla u}\ 
\hbox{ and } \, 
\Delta _1 ^N u := |\nabla u|\, {\rm div} \Big ( \frac{\nabla u} {|\nabla u|} \Big )\, . $$ 
Let us point out that solutions to \eqref{f:pb1} are in general not classical, i.e. of class $C^2$, but have to be understood as viscosity solutions and these are defined in Section \ref{sec:proofs}. 

The normalized $p$-Laplacian has recently received increasing attention, partly because of its application in image processing \cite{K0,Does} and in the description of tug-of-war games (see \cite{PSSW1, PSSW2}).  
Without claiming to be complete we list \cite{BK18,CFd,CFe,CFf,CF7,EKNT, JK, Juut07, KH, K11, kuhn , MPR1, MPR2} for some related works. 

 Following Berestycki, Nirenberg, and Varadhan \cite{BNV},  
in the paper \cite{BiDe2006} (where actually they deal with a wider class of operators),  
Birindelli and Demengel introduced  the  {\it first eigenvalue of $\nplap$ in $\Omega$} as
 $$\overline \lambda _p (\Omega) := \sup \Big \{ \lambda_p \in \R \ :\  \exists u >0 \text{ such that } \ 
 \nplap u + \lambda _p u \leq 0\ \text{ in the viscosity sense}
 \Big \}\,.$$ 
They proved that calling it first eigenvalue is justified, see \cite[Theorems 1.3 and 1.4]{BiDe2006}.  
In particular they showed that there exists a positive eigenfunction associated with
$\overline \lambda _p (\Omega)$. In other words for $\lambda _p = \overline \lambda _p (\Omega)$
problem \eqref{f:pb1} admits a positive viscosity solution.  
They also posed the open problem to determine whether $\overline \lambda _p (\Omega)$ is  simple.
We show that the answer is affirmative.  More precisely, we prove:

\begin{theorem}\label{t:simple}
Let $\Omega$ be an open bounded domain in $\R ^n$, with $\partial \Omega$ smooth and connected. 
If $u$ and $v$ are two positive eigenfunctions associated with 
$\overline \lambda _p (\Omega)$, then $u$ and $v$ are proportional, that is there exists $t \in \R _+$ such that
$u = tv $ in $\Omega$.
 \end{theorem}

Here and in the following, $\partial\Omega$ smooth means that
it is of class $C^{2,\alpha}$.  
Theorem \ref{t:simple} has the following  immediate consequence: 
 
  \begin{corollary}\label{symmetry}
 Let $\Omega$ be an open bounded domain in $\R ^n$, with $\partial \Omega$ smooth and connected. If $\Omega$ is invariant under elements from a symmetry group such as reflections or rotations, then so are the first eigenfunctions of the normalized $p$-Laplace operator.
 \end{corollary}
 
In order to obtain Theorem \ref{t:simple} we follow the approach used by Sakaguchi in \cite{Sak}. 
In particular, it will be clear by inspection of the proof that this method does not work if one drops the assumption that 
$\partial \Omega$ is connected. It is conceivable that the result continues to be true for more general domains, as it is known in the literature for other kinds of operators at least in dimension two  (see for instance \cite[Theorem 4.1]{BirDem2010}). 
 
 As a fundamental preliminary tool, our proof of Theorem \ref{t:simple} exploits a Hopf type lemma (see Lemma \ref{t:hopf})  and, incidentally, it requires also the strict positivity of the eigenvalue. The latter can be easily established by comparison with the behaviour on balls (see Lemma \ref{l:ball} and Lemma \ref{l:positivity}).  
  In fact, the observation that $\overline\lambda_p(\Omega_1)\geq\overline\lambda_p(\Omega_2)$ for $\Omega_2\subset\Omega_1$ leads to the bounds 
 \begin{equation}\label{ballcomp}
\overline\lambda_p(B_R)\leq \overline\lambda_p(\Omega)\leq\overline\lambda_p(B_\rho),
\end{equation}
where $\rho$ and $R$ denote inradius and outer radius of $\Omega$, see the recent papers \cite{blanc, KH}.
These bounds are sharp if $\Omega$ is a ball, but they are far from optimal if $R-r$ becomes large, e.g. for slender ellipsoids. On the other hand, the problem of finding more accurate bounds for the eigenvalue seems to be an interesting and mostly unexplored question.   In this respect \eqref{ballcomp}  is complemented by the following lower estimate for $\overline \lambda _p (\Omega)$ in terms of the Lebesgue measure of $\Omega$.
 
\begin{theorem}\label{t:FK}
 For every open bounded  domain $\Omega$ in $\R ^n$   we have the lower bound 
$$\overline \lambda _p (\Omega) \geq K _{n, p} |\Omega| ^ { - 2/n} \, , $$
with 
\begin{equation}\label{f:Knp} 
K_{n,p}:= \frac{ \big ( n[(p-1) \wedge 1 ] ) ^ 2 } {p(p-1) }   \,
4^{-1+1/n}\, \pi^{1+1/n}\, \Gamma\left(\frac{n+1}{2}\right)^{-2/n}\,.
\end{equation}

\end{theorem}
 
The proof of Theorem \ref{t:FK}  will be obtained by the Alexandrov--Bakelman--Pucci method, as addressed by Cabr\'e in \cite{C15} (see also \cite{CDDM}). Unfortunately, it seems to be an intrinsic drawback of this  approach to provide  a non-optimal estimate. Actually it is natural to conjecture that, as in case of the well-known Faber-Krahn inequality for the $p$-Laplacian, the product  $\overline \lambda _p (\Omega) |\Omega| ^ { 2/n}$ should be minimal on balls. 
 In other words, the optimal lower bound expected for the product $\overline \lambda _p (\Omega) |\Omega| ^ {  2/n}$ is the constant 
 $K^*_{n,p}:= \overline \lambda _p (B)  |B| ^ {2/n} $. 
Notice that due to the scaling invariance $B$ can be an arbitrary ball here. 
To prove such an optimal bound seems to be a very interesting and delicate problem. The symmetrization technique usually employed to prove the Faber-Krahn inequality for the $p$-Laplacian does not work here because the normalized $p$-Laplacian operator does not have a variational nature.  
 
 To demonstrate that \eqref{f:Knp} is not optimal for balls let us sketch a quick comparison between the values of $K _{n,p}$ and $K ^* _{n,p}$. 
  Clearly, by Theorem \ref{t:FK}, the quotient $K^*_{n,p} / K _{n, p}$ is larger than or equal to $1$.  
 In order to evaluate the presumed accuracy of our estimate, one can evaluate how far it is from $1$. 
 As shown in Lemma \ref{l:ball} below, we have  
\begin{equation}\label{f:Knp2} 
K^*_{n,p} =  \frac{\pi(p-1)}{p} \Gamma \Big ( 1 + \frac{n}{2} \Big ) ^ { -2/n} \big ( \mu _1 ^ { (- \alpha ) }  \big ) ^ 2 \,,
\end{equation}
where $\mu_1^{(-\alpha)}$ denotes the first zero of the Bessel function 
$J _{-\alpha}$, with $\alpha = \tfrac{p-n}{2(p-1)}$.
The plots in Figure \ref{fig:comparison} left and right, obtained with Mathematica,  represent this ratio in two and three dimensions as a function of $p$.
Observe that both maps
 $$p \mapsto  g_2 (p):= \frac{K^*_{2,p}}{K _{2, p}} \, , \qquad p \mapsto g_3 (p):= \frac{K^*_{3,p}}{K _{3, p}} \,$$ 
 turn out to be minimal at $p = 2$, with 
 $$g_2 (2) \approx 1.446\, , \qquad g_3 (2)\approx 1.561\,.$$
This shows that the constant $K_{n,p}$  in Theorem \ref{t:FK} is not optimal, not even in the linear case $p=2$.
\begin{figure}[ht]
\begin{minipage}{0.5\linewidth}
\centering
\includegraphics[height=4cm]{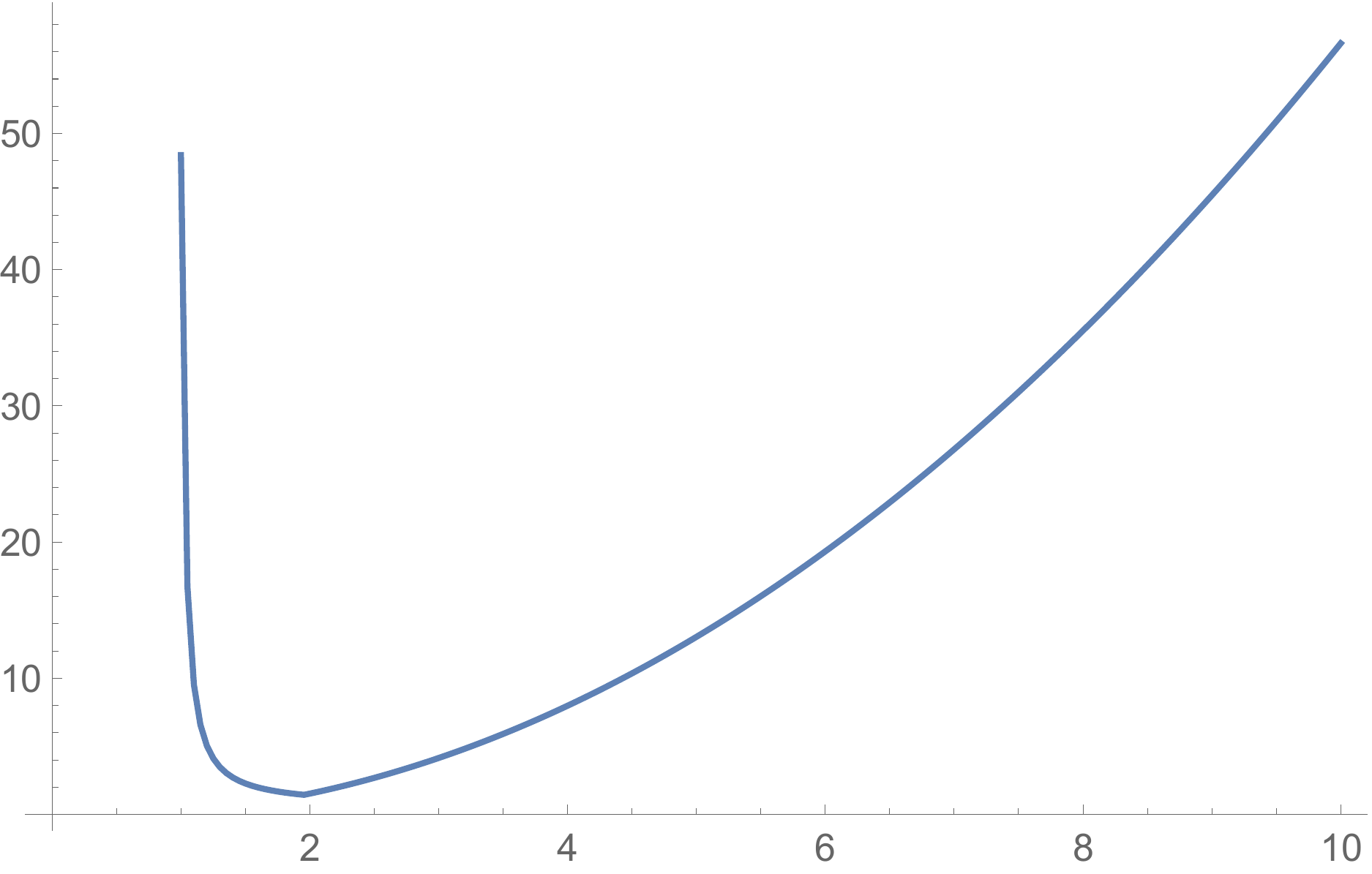}
\end{minipage}%
\begin{minipage}{0.5\linewidth}
\centering
\includegraphics[height=4cm]{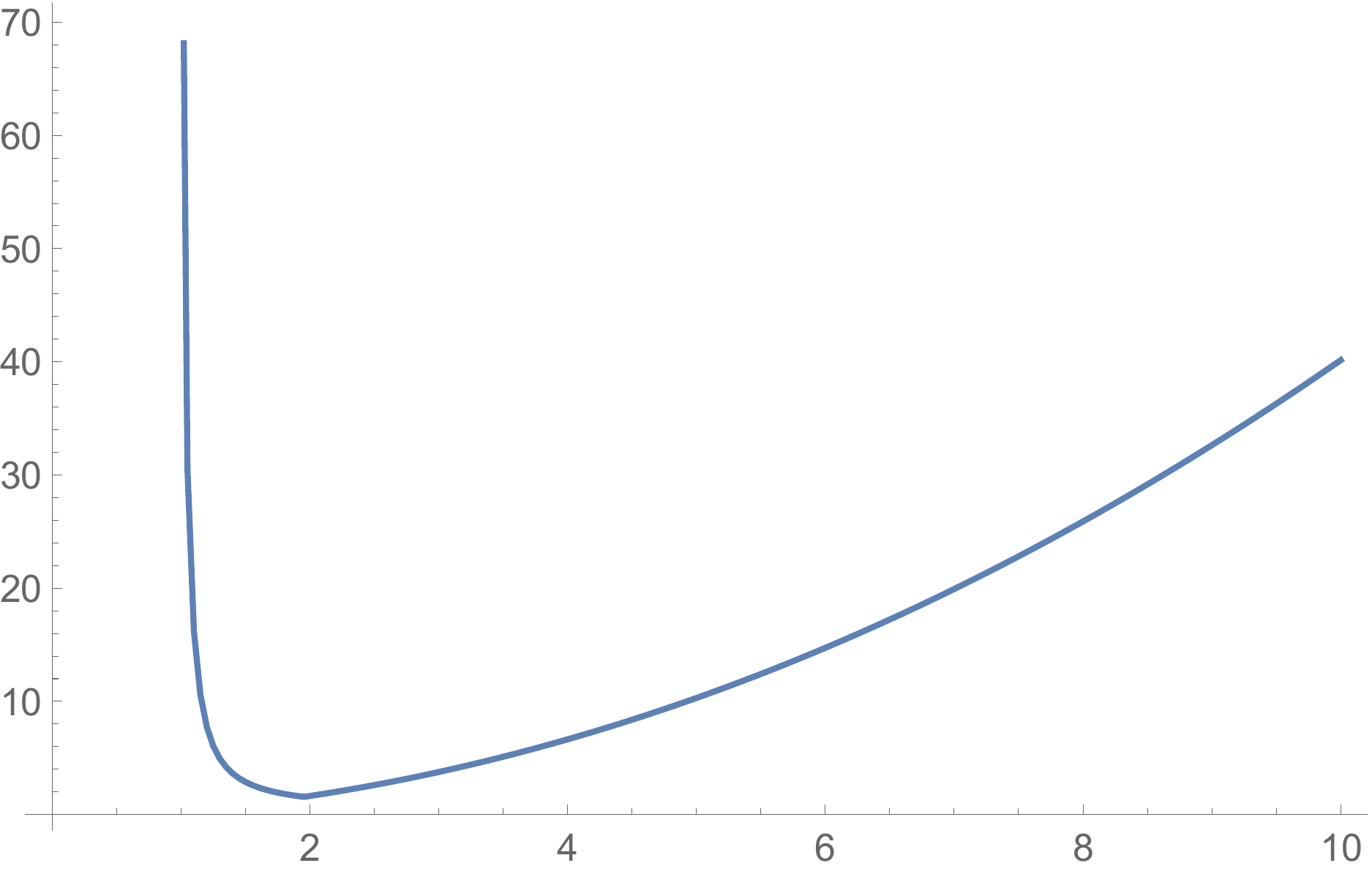}
\end{minipage}
\caption{Plots of $g_2(p)$ and $g_3(p)$}   
\label{fig:comparison}   
\end{figure}

The proofs of Theorems \ref{t:simple} and \ref{t:FK} are given in Section \ref{sec:proofs} below, 
after recalling the definition of viscosity solution to problem \eqref{f:pb1} and providing some preliminary results.

\section{Proofs}\label{sec:proofs}

In the notation of viscosity theory, the equation $- \nplap u = \lambda _p u $  can be rewritten as
\begin{equation}\label{F_p}
F_p ^N(\nabla u, \nabla ^ 2 u ) = \lambda_p u\, , 
\end{equation}
 where $F_p^N$  is defined on $(\R^n\setminus \{ 0 \})\times S(n)$ and $S(n)$ denotes the space of $n\times n$ symmetric matrices, with
\begin{equation}\label{f:Fp}
F_p^N(\xi, X) := - \frac{p-2}{p} |\xi|^{-2} \pscal{X\xi}{\xi} - \frac{1}{p}\, \Trace(X)\quad \forall \xi \in \R ^n \setminus \{ 0 \}\,,\ X \in S(n). 
\end{equation}
At $\xi=0$ the function $F_p^N$ is discontinuous.  In this case, following \cite{CIL} we request from a viscosity solution of \eqref{F_p}
that it is a viscosity subsolution of $(F^{N}_{p})_*(Du,D^2u)=\lambda_pu$ and a viscosity supersolution of $(F^{N}_p)^*(Du,D^2u)=\lambda_pu$.
 Here $(F^{N}_p)^*$ is the upper semicontinuous hull and $(F^N_{p})_*$ is the lower semicontinuous hull of $F^N_p$.

Now since $F_p^N$ is given by
\begin{equation*}
F_p^N(\xi,X)=
-\frac{1}{p}\left(\delta_{ij}+(p-2)\frac{\xi_i\xi_j}{|\xi|^2}\right)X_{ij}
 \hbox{ for }\xi\not=0 
\end{equation*}
we have to compute its semicontinuous limits as $\xi\to 0$. Each
symmetric matrix $X$ has real eigenvalues, and we order them according to magnitude as $\lambda_1(X)\leq\lambda_2(X)\leq\cdots\leq\lambda_n(X)$. 
Then a simple calculation shows that
\begin{equation}\label{F_p^N_*}
{(F_{p}^N)_*}(0,X)=\begin{cases}
\,-\, \frac{1}{p}\sum_{i=1}^{n-1}\lambda_i-\frac{p-1}{p}\lambda_n & \hbox{ if } p\in[2,\infty],\\
-\frac{1}{p}\sum_{i=2}^{n}\lambda_i-\frac{p-1}{p}\lambda_1 &\hbox{ if } p\in[1,2],\\
\end{cases}
\end{equation}
and
\begin{equation}\label{F_p^N^*}
(F_p^N)^*(0,X)=\begin{cases}
-\frac{1}{p}\sum_{i=2}^{n}\lambda_i\,-\frac{p-1}{p}\lambda_1 &\hbox{ if } p\in[2,\infty],\\
-\frac{1}{p}\sum_{i=1}^{n-1}\lambda_i-\frac{p-1}{p}\lambda_n & \hbox{ if } p\in[1,2].
\end{cases}
\end{equation}
In \cite{Bru} these bounds for the normalized $p$-Laplacian are called dominative and submissive $p$-Laplacians and studied in more detail. Anyway, the above considerations serve as a motivation for the following

\begin{definition}\label{def:vs}
Given a symmetric matrix $A\in S(n)$, we denote by $M (A)$ and $m (A)$ its greatest and smallest eigenvalue. 

\smallskip
-- An upper semicontinuous function $u :\Omega \to \R$ is  a viscosity subsolution of $- \nplap u = \lambda_p u$
in $\Omega$ if, for every point $x$ in $\Omega$ and every smooth function $\varphi$  which touches $u$ from above at $x$ (and for which $u - \varphi$ attains a local maximum at $x$) it holds
$$
\begin{cases}
- \nplap \varphi (x) \leq \lambda_p \varphi  (x)  & \text{ if } \nabla \varphi (x) \neq 0 
\\  \medskip 
- \frac{1}{p} \Delta \varphi (x) - \frac{(p-2) }{p} M (D ^ 2 \varphi (x) ) \leq \lambda_p \varphi (x) & \text { if  } \nabla \varphi (x) = 0 \text{ and } p \geq 2 
\\  \medskip 
- \frac{1}{p} \Delta \varphi (x) - \frac{(p-2) }{p} m (D ^ 2 \varphi (x) ) \leq \lambda_p \varphi (x) & \text { if  } \nabla \varphi (x) = 0 \text{ and } p \leq 2 . 
\end{cases}
$$ 
-- A lower semicontinuous function $u :\Omega \to \R$ is  a viscosity supersolution of $- \nplap u = \lambda_p u$
in $\Omega$ if, for every point $x$ in $\Omega$ and every smooth function $\varphi$  which touches $u$ from below at $x$ (and for which $u - \varphi$ attains a local minimum at $x$) it holds
$$
\begin{cases}
- \nplap \varphi (x) \geq \lambda_p \varphi  (x)  & \text{ if } \nabla \varphi (x) \neq 0 
\\  \medskip 
- \frac{1}{p} \Delta \varphi (x) - \frac{(p-2) }{p} m (D ^ 2 \varphi (x) ) \geq \lambda_p \varphi (x) & \text { if  } \nabla \varphi (x) = 0 \text{ and } p \geq 2 
\\  \medskip 
- \frac{1}{p} \Delta \varphi (x) - \frac{(p-2) }{p} M (D ^ 2 \varphi (x) ) \geq \lambda_p \varphi (x) & \text { if  } \nabla \varphi (x) = 0 \text{ and } p \leq 2 . 
\end{cases}
$$ 
-- A continuous function $u :\Omega \to \R$ is a  viscosity supersolution to $- \nplap u = \lambda_p u$  if it is both a viscosity supersolution and a viscosity subsolution. 
\end{definition}

\begin{remark}\label{r:propF}
For later use we mention that the function $F_p^N$  satisfies the following identities: 
\begin{itemize}
\item[(i)] $-F_p^N (t \xi, \mu X) =  - \mu F_p^N ( \xi, X)\quad \forall t \in \R \setminus \{ 0 \}$,  $\xi \in \R ^n\setminus \{ 0 \} $, $\mu \in \R$, and $X\in S(n)$.
\item[(ii)]  $-F_p^N (\xi, X ) \leq 0$ for any $\xi \in \R ^n \setminus \{ 0 \}$ and $X\in S(n)$ with $X\leq 0$. 
This follows from  \eqref{f:Fp}, since the eigenvalues $\lambda_i(X)$ are assumed nonpositive.

\item[(iii)]  As a consequence of \eqref{f:Fp}, \eqref{F_p^N_*} and \eqref{F_p^N^*},
for every $\xi\in\R^n$ and $X\in S(n)$ we have that
\begin{gather*}
\frac{(p-1)\wedge p}{p}\,
\Trace(X)
\leq
-F_p^N(\xi,X)
\leq
\frac{(p-1)\vee p}{p}\,
\Trace(X)\,,
\qquad
\text{if}\  X\geq 0,
\\
\frac{(p-1)\vee p}{p}\,
\Trace(X)
\leq
-F_p^N(\xi,X)
\leq
\frac{(p-1)\wedge p}{p}\,
\Trace(X)\,,
\qquad
\text{if}\ X\leq 0.
\end{gather*}
\end{itemize} 
\end{remark}

\bigskip 
 For $x \in \R ^n$ and $R>0$, we denote by $B _R (x)$ the open ball of radius $R$ centred at $x$. We also set for brevity $B _ R := B _ R (0)$.

\begin{lemma}[First eigenvalue of the ball]\label{l:ball} 
For any $p \in (1, + \infty)$, we have
$$\overline \lambda _p ( B _R) = \frac{p-1}{p} \Big ( \frac{\mu_1^{(-\alpha)}} {R} \Big ) ^ 2   = K ^ *_ {n,p} |B _R| ^ {-2/n}\, ,  $$
where $\mu_1^{(-\alpha)}$ denotes the first zero of the Bessel function 
$J _{-\alpha}$, for $\alpha = \tfrac{p-n}{2(p-1)}$  (and the constant $K ^ * _{n,p} $ is defined in \eqref{f:Knp2}). 
 \end{lemma} 

\begin{proof}
Set $\lambda_p (R):= \frac{p-1}{p} \Big ( \frac{\mu_1^{(-\alpha)}} {R} \Big ) ^ 2$. We first prove that $\overline \lambda _p ( B _R) \geq \lambda_p(R)$.  By definition, this amounts to show that problem \eqref{f:pb1} admits a positive viscosity subsolution when $\lambda_p = \lambda_p (R)$. We search for a radial solution and make the ansatz $u (x) = g (|x|)$. In terms of the function $g = g (r)$, problem \eqref{f:pb1} can be written as (see \cite{KKK})
\begin{equation}\label{f:pbradial} 
\begin{cases}
\displaystyle -  g''(r) - \Big ( \frac{n-1}{p-1} \Big )  \frac{g'(r)}{r} =\Big (  \frac{p}{p-1}  \Big ) \lambda _p \, g (r) & \text{ on } (0, R)
\\ \noalign{\medskip} 
g (R) = 0  & 
\\  \noalign{\medskip} 
g' (0) = 0 \,. & 
\end{cases}
\end{equation} 
For $p=2$ the left hand side in the differential equation is just the classical Laplacian, evaluated in polar coordinates for $g(|x|)$. For other $p$ it can be interpreted as a linear Laplacian in a fractional dimension. This was done in \cite{KKK}, and a full spectrum and orthonormal system of radial eigenfunctions was derived. The first eigenfunction is a (positive) multiple of $r^\alpha J_{-\alpha}(\mu_1^{(-\alpha )}\frac{r}{R})$. This function is positive in $B_R$.

Finally, let us show that the equality $\overline \lambda _p ( B _R)=  \lambda _p(R)$ holds. For this we use an idea from \cite{MPR2}, there given for $p>n$. Assume by contradiction that  
$\overline \lambda _p ( B _R)>  \lambda _p(R)$. Choose $\rho \in (0, R)$ such that 
$\overline \lambda _p ( B _R)> \lambda_p (\rho) > \lambda_p(R)$,  and
let $g _\rho$ be a positive solution to problem 
\begin{equation}\label{f:pbradial2} 
\begin{cases}
\displaystyle -  g''(r) - \Big ( \frac{n-1}{p-1} \Big )  \frac{g'(r)}{r} =\Big (  \frac{p}{p-1}  \Big ) \kappa _p (\rho) \, g (r) & \text{ on } (0, \rho)
\\ \noalign{\medskip} 
g (R) = 0  & 
\\  \noalign{\medskip} 
g' (0) = 0 \,. & 
\end{cases}
\end{equation} 
Then 
the function $w$ defined on $B _R$ by
$w (x) = g _\rho (|x| )$ if $|x| \leq \rho$ and $0$ otherwise turns out to satisfy
$- \nplap w \leq \lambda _p (\rho) w$ in $B _R$  and $w \leq 0$ on $\partial B _R$. 
In view of Remark \ref{r:propF} (i) and (ii), the  operator $\nplap$ 
 satisfies the assumptions of the 
comparison result stated in \cite[Theorem 1.1]{BiDe2006}.  
We infer that $w\leq 0$ in $B_R$, a contradiction.
\end{proof}

 \bigskip
 \begin{lemma}[Positivity of the eigenvalue] \label{l:positivity} For every open bounded domain  $\Omega \subset \R ^n$, we have $\overline \lambda _p (\Omega) >0$. 
 \end{lemma} 

\begin{proof} 
From its definition, it readily follows that $\overline \lambda _p$ is monotone decreasing under domain inclusion, {\it i.e.} $\overline \lambda _p (\Omega_1 ) \geq \overline \lambda _p (\Omega_2 )$ if $\Omega _1 \subseteq \Omega _2$. In particular, for every open bounded domain $\Omega$, we have 
$\overline \lambda _p (\Omega) \geq \overline \lambda _ p  ( B _ R)$, where $R = R (\Omega) = \inf \big\{ r >0 \, :\, \Omega \subset B _ r (x)\text{ for some } x \big \}$.   Invoking Lemma \ref{l:ball}, we obtain the positivity of $\overline \lambda _p (\Omega)$. 
\end{proof}

\bigskip

In the following Lemma we do not assume differentiability of $u$ on the boundary. Nevertheless we can bound the difference quotient in interior normal direction from below.

\begin{lemma}[Hopf type Lemma]\label{t:hopf}
Assume that $\Omega\subset\R^n$ satisfies a uniform interior sphere
condition,
and let $u\in C(\overline{\Omega})$ be a positive viscosity supersolution
of $-\nplap u = 0$  in $\Omega$ such that $u=0$ on $\partial\Omega$.
Then there exists a constant $\kappa > 0$ such that for any $y\in \partial\Omega$
\begin{equation}\label{f:hopf}
\liminf_{t\to 0+} \frac{u(y-t\nu(y))}{t} \geq \kappa.
\end{equation}
Here $\nu$ denotes the unit outer normal to $\partial \Omega$, 
\end{lemma}

\begin{proof}
This follows from realizing that the normalized $p$-Laplacian satisfies the assumptions in \cite[Theorem 1]{BDL}. 
\end{proof}

\bigskip

{\bf Proof of Theorem \ref{t:simple}}. 
Let $u$ and $v$ be two positive eigenfunctions associated with $\overline \lambda _p (\Omega)$. 
Inspired by the appendix in \cite{Sak} we set
\[
\begin{array}{ll}
& a:= \sup \Big \{ t \in \R \ : \ u - tv >0 \text{ in } \Omega \Big \}
\\
\noalign{\medskip} 
& b:= \sup \Big \{ t \in \R \ : \ v - tu >0 \text{ in } \Omega \Big \}\,.
\end{array}
\] 
 Clearly, we have 
 \begin{equation}\label{f:pos} 
 u - a v \geq 0 \quad  \text{ and } \quad v-bu \geq 0 \qquad \text{ in } \Omega \,.
 \end{equation}
We claim that $a$ and $b$ are strictly positive.  
Indeed, the functions $u$ and $v$ are of class $C ^ {1, \alpha}$ up to the boundary (see \cite[Proposition 3.5]{BirDem2010} or  \cite[Theorem 1.1]{APR17}). Then, applying Lemma \ref{t:hopf} 
to $u$ and $v$, we see that
\begin{equation}\label{f:hopf2}
\frac{\partial u }{\partial \nu} <0 \quad \text{ and } \quad \frac{\partial v }{\partial \nu} <0 \qquad \text{ on } \partial \Omega\,.
\end{equation} 
Hence, for $t \in \R _+$ small enough,  $\frac{\partial}{\partial \nu} (u - t v)$ is strictly negative on $\partial \Omega$, so that  there exists $\overline t>0$   and a neighbourhood $\mathcal U$ of $\partial \Omega$ such that 
$u- \overline tv >0$ in $\mathcal U$.  It follows that
$$u- m  v >0 \text{  in } \Omega \quad \text{ for } 
m<\min \Big \{ \overline t, \frac{\min _{\Omega \setminus \overline { \mathcal U } } u } {\max_{\Omega \setminus \overline {\mathcal U } } v} \Big \}\,.$$  
Thus $a \geq m >0$. Arguing in the same way with $u$ and $v$ interchanged we obtain $b>0$, and our claim is proved. 

Now, to obtain the result, we are going to show that there exists a neighbourhood $\mathcal V$ of $\partial \Omega$ such that
\begin{equation}\label{f:neig}
u - a v = 0 \quad \text{ and } \quad v - b u = 0 \quad \text{ in } \mathcal V\,.
\end{equation}
This implies $u - (ab) u = 0$ in $\mathcal V$ and, in view of the condition $u >0 $ in $\Omega$, $b = a ^ {-1}$. 
 The latter equality, combined with \eqref{f:pos}, implies $u - av = 0$  in $\Omega$ as required. 

Let us show how to obtain the first equality in \eqref{f:neig}, the derivation of the second one is  completely analogous. 

By the regularity of $\partial \Omega$, its unit outer normal $\nu$ can be extended to a smooth unit vector field, still denoted by $\nu$, defined in an open connected neighbourhood of $\partial \Omega$. 
Then, by \eqref{f:hopf2} and the $C ^ 1$ regularity of $u$ and $v$ on $\overline \Omega$, we infer that there exist $\delta >0$ and an open connected neighbourhood $\mathcal V$ of $\partial \Omega$ such that 
\begin{equation}\label{f:nondeg}
\frac{\partial u }{\partial \nu} <-\delta \quad \text{ and } \quad \frac{\partial v }{\partial \nu} <- \delta \qquad \text{ in } \overline {\mathcal V}\,.
\end{equation}
This implies first of all that the PDE solved by $u$ and $v$ is nondegenerate in $\mathcal V$, which in turn, by standard elliptic regularity (see \cite{GT}) yields that $u$ and $v$ are of class $C ^ \infty$ in $\mathcal V$.  Moreover, from the inequality
$$0 \leq \overline \lambda _p u - \overline \lambda _p (av) \qquad \text{ in } \Omega$$
we infer that
$$0 \leq - \nplap u -  \big ( - \nplap (av) \big )  = L_p (u - av)  \qquad \text{ in } \Omega\, , $$
 where $L_p w = \sum _{i,j= 1}^n c_{ij} w _{x_i x_j} + \sum _{i = 1} ^ n d _i w _{x_i}$ is the linear operator defined by
 $$\begin{array}{ll}
 & \displaystyle  c _{ij}:= \int _0 ^ 1 \frac{\partial F_p^N }{\partial X _{ij} } \big ( s  \nabla u+ ( 1-s ) \nabla v,  s  \nabla^2 u+ ( 1-s ) \nabla^2  v  \big ) \, ds 
 \\
 \noalign{\bigskip} 
 & \displaystyle  d _{i}:= \int _0 ^ 1 \frac{\partial F_p^N }{\partial \xi _{i}} \big ( s  \nabla u+ ( 1-s ) \nabla v,  s  \nabla^2 u+ ( 1-s ) \nabla^2  v  \big ) \, ds \,.
 \end{array}
 $$ 
 In particular, since 
 $$ \frac{\partial F_p^N }{\partial X _{ij} } = - \frac{p-2}{p}\frac{1}{ |\xi| ^ 2} \xi _i \xi _j - \frac{1}{p} 
  \delta _{ij}\,$$ 
  and, from \eqref{f:nondeg}, 
  $$\forall s \in [0,1], \quad s  \frac{\partial u}{\partial \nu} +  (1-s) \frac{\partial(a v) }{\partial \nu}  \leq  - \min \{ \delta, a \delta \} <0 \quad \text{ in } \mathcal V\, ,$$
  we see that $L_p$ is {\it uniformly elliptic} in 
the connected set $\mathcal V$. 
 Then, to achieve our proof, it is enough to show that there exists some point $x^* \in \mathcal V$ where the function $
 u-av$ vanishes. Indeed, if this is the case, we have:
 $$
 \begin{cases}
L_p (u-av) \geq 0   & \text{ in } \mathcal V
\\ 
\noalign{\medskip}
u - av \geq 0 & \text{ in } \mathcal V
\\
\noalign{\medskip}
(u-av) (x^*) = 0\,. 
\end{cases} 
 $$ 
 By the strong maximum principle for uniformly elliptic operators  \cite[Theorem 3.5]{GT}, it will follow that
 $u-av \equiv 0$ in $\mathcal V$ as required. 
 We point out that, without the connectedness of $\partial \Omega$ (and hence of $\mathcal V$), the two equalities in \eqref{f:neigh} might be obtained in two, a priori distinct, connected components of $\mathcal V$, and this would not be sufficient to infer that $u$ and $v$ are proportional.

 To conclude, let us now show that $u-av$ vanishes at some point $x^*$ in $\mathcal V$. As an intermediate step we notice that the function $u-av$ must vanish at some point $\overline x$ in $\Omega$. Otherwise, we would have:  
 $$
 \begin{cases}
L_p (u-av) \geq 0   & \text{ in } \mathcal V
\\ 
\noalign{\medskip}
u - av > 0 & \text{ in } \mathcal V
\\
\noalign{\medskip}
u-av \equiv  0 & \text{ on } \partial \Omega\,. 
\end{cases} 
 $$ 
By applying Hopf's boundary point lemma for uniformly elliptic operators \cite[Lemma 3.4]{GT}, we infer that $\frac{\partial}{\partial \nu} (u - av) < 0$ on $\partial \Omega$. By continuity, this inequality, combined with the strict one $u-av >0 $ in $\Omega$ that we are assuming by contradiction, implies that $u - (a + \eta) v >0$  in $\Omega$ for some $\eta>0$. But this contradicts the definition of $a$.

Now, we choose an open bounded set $\omega$ with smooth boundary such that
\[
\overline \omega \subset\Omega\, , \quad \overline x \in \omega\, , \quad \partial \omega \subset \mathcal V\,.
\] 
We assert that there is a point $x ^* \in \partial \omega $ where $u-av$ vanishes (and this point does the job since 
$\partial \omega \subset \mathcal V$).  Assume the contrary. Then by continuity we have 
$u - av \geq \varepsilon >0$ on $\partial \omega$ for some 
$\varepsilon >0$.  Then the two functions $u$ and $w:= av +\varepsilon$  satisfy
\[
\begin{cases}
- \nplap u  = \overline \lambda _p u  \geq \overline \lambda _p (av)= - \nplap w & \text{ in } \omega
\\ 
u \geq w & \text{ on } \partial \omega\,.
\end{cases}
\]    
 In view of Lemma \ref{l:positivity}, the continuous function $f:=\overline \lambda _p u$ is strictly positive in $\omega$.  Now we can apply the comparison principle proved in \cite[Thm.~2.4]{LuWang2008},
and we infer that
\[
u \geq w \quad \text{ in } \omega\,.
\] 
In particular, since $\omega$ contains the point $\overline x$, we have 
$$u (\overline x ) \geq w (\overline x) = av (\overline x) +\varepsilon\,,$$ 
which gives a contradiction since $u (\overline x ) = av (\overline x)$.
\qed

\bigskip

 \bigskip
 In order to prove Theorem \ref{t:FK}, we need some preliminary results. 
 
 Let $u$ be  a positive eigenfunction associated with $\overline \lambda _p (\Omega)$. 
The approximations of 
$u$ via supremal convolution are defined for $\e >0$ by
\begin{equation}\label{f:ue}
u ^ \e (x) := \sup _{y \in \Omega} \Big \{ u (y)  - \frac { |x-y| ^ 2 }{2 \e}  \Big \} \qquad \forall x \in \Omega\,.
\end{equation}

Let us start with a preliminary lemma in which we recall
some basic well-known properties of the functions $u ^ \e$. 
To fix our setting let us define
\[
\rho(\e) := 2 \sqrt{\e \, \|u\|_\infty},
\qquad
 \Omega^{\rho(\e)} := \{x\in\Omega:\ d_{\partial\Omega}(x) > \rho(\e)\}\,,
\]
then for every \(x \in \Omega^{\rho(\e)}\) the supremum in \eqref{f:ue}
is attained at a point \(y_\e(x) \in \overline{B}_{\rho(\e)}(x)\subset\Omega\).
Thus, setting
\begin{equation}\label{f:Ue} 
U _\e:= \big \{ x \in \Omega \ :\ u (x) > \e \big \} \, , \qquad 
A_\e := \big \{ x \in U _\e \ :\ d_{\partial U _\e}(x) > \rho(\e) \big \}\, ,
\end{equation}
so that by definition
\begin{equation}\label{d:ue2}
u ^ \e (x) = 
u(y_\e(x)) - \frac{|x-y_\e(x)|^2}{2\e} =
\sup _{y \in U _\e} \Big \{ u (y)  - \frac { |x-y| ^ 2 }{2 \e}  \Big \} 
\qquad \forall x \in \overline A_\e\,. 
\end{equation}
In what follows, we shall always assume that $\e \in (0, 1)$ is small enough
to have $A_\e \neq \emptyset$.
Moreover, let us define
\begin{equation}\label{f:omegae}
m_\e := \max_{\partial A_\e} u^\e,
\qquad
\Omega_\e := \{x\in A_\e : \ u^\e (x) > m_\e \}\,.
\end{equation}

\begin{lemma}\label{l:approx1} 
Let
$u$ be a positive eigenfunction associated with $\overline \lambda _p (\Omega)$, let 
   $u ^\e$ be its supremal convolutions according to \eqref{f:ue}, and let  $\Omega_\e$ be the domains defined in \eqref{f:omegae}. Then:  
\begin{itemize}
\item[(i)] $u ^ \e$ is semiconvex in $\Omega _\e$;

\smallskip
\item[(ii)] $u ^\e$ is a viscosity sub-solution to $-\nplap  u -\overline \lambda _p (\Omega) u = 0$ in $\Omega_\e$;

\smallskip
\item[(iii)] as $\e \to 0 ^+$,  $u ^ \e$ converge to $u$ uniformly in $\overline \Omega$. 
Hence $m_\e \to 0$ and $\Omega_\e$ converges to $\Omega$
in Hausdorff distance; 

\smallskip
\item[(iv)] 
as $\e \to 0 ^+$,  $\nabla u ^ \e \to \nabla u$ locally uniformly in $\Omega$.
\end{itemize}
\end{lemma}

\begin{proof}
(i)  We have $u ^ \e = - ( -u ) _\e$, where $( -u ) _\e$ is the infimal convolution defined by
\[ 
( -u ) _\e (x) := \inf _{y \in U _\e} \Big \{-  u (y)  + \frac { |x-y| ^ 2 }{2 \e}  \Big \} \qquad \forall x \in \Omega _\e\,.
\]
{}From \cite[Proposition 2.1.5]{CaSi}, it readily follows that $(-u) _\e$ is semiconcave on $\Omega _\e$, and hence that 
$u ^ \e$ is semiconvex on $\Omega _\e$.  
 
 (ii) The notion of of viscosity subsolution according to Definition \ref{def:vs} can be reformulated by asking that, for every $x \in \Omega$ and every $(\xi, X)$ in the second order 
superjet $J ^ {2, +} _\Omega u (x)$   (classically defined as in \cite{CIL}), it holds 
$$
\begin{cases}
F_p^N(\xi, X) \leq \lambda_p u  (x)  & \text{ if } \xi \neq 0 
\\  \medskip 
- \frac{1}{p} \Trace(X)  - \frac{(p-2) }{p} M (X) \leq \lambda_p u (x) & \text { if  } \xi = 0 \text{ and } p \geq 2 
\\  \medskip 
- \frac{1}{p} \Trace(X)  - \frac{(p-2) }{p} m (X) \leq \lambda_p u (x) & \text { if  }\xi = 0 \text{ and } p \leq 2 . 
\end{cases}
$$ 
Then, in order to prove (ii), it is enough to show that, for every fixed point $x \in \Omega _\e$, any pair  $(p, X) \in  J  ^ {2, +} _{\Omega _\e} u ^ \e (x)$ belongs to $J ^ {2, +} _{\Omega} u (y)$ for some other point $y \in \Omega _\e$. In fact, the so-called magic properties of supremal convolution ({\it cf.} \cite[Lemma A.5]{CIL}) assert precisely that any 
$(p, X) \in  J  ^ {2, +} _{\Omega _\e} u ^ \e (x)$ belongs to
$J ^ {2, +} _{\Omega _\e} u (y)$, where $y$ is a point at which the supremum which defines
$u ^ \e (x)$ is attained.  Since $y \in U _\e \subset \Omega _\e$, it holds $J ^ {2, +} _{\Omega} u (y)= J  ^ {2, +} _{\Omega _\e} u ^ \e(x)$. 

(iii) For these convergence properties we refer to \cite[Thm.\ 3.5.8]{CaSi},
\cite[Lemma 4]{CFd}. 

(iv)
Since $u\in C^1(\Omega)$, this property follows
from \cite[Lemma 10]{CF7}.
\end{proof}

\bigskip\bigskip

\begin{lemma}\label{l:approx2} 
Let
$u$ be a positive eigenfunction associated with $\overline \lambda _p (\Omega)$,  let 
   $u ^\e$ be its supremal convolutions according to \eqref{f:ue}, and let $\Omega_\e$ be the domains defined   in \eqref{f:omegae}.   Let $v ^ \e$ be the continuous functions defined by 
\begin{equation}\label{f:ve}
v ^ \e (x) := \begin{cases}
\log ( u ^ \e ) & \text{ if } x \in \Omega _\e
\\ 
\log (m _ \e) & \text{ if } x \in \R ^n \setminus \Omega _\e
 \end{cases} 
 \end{equation}
and, for $\sigma >0$, let $\Gamma _\sigma (v ^ \e)$ be the concave envelope of $v ^ \e$ on the set
   \begin{equation}\label{f:neigh}
   (\Omega _\e ^* ) _\sigma:= \Big \{ x \in \R ^n \ :\ {\rm dist} (x, \Omega _\e ^ *) \leq \sigma \Big \}\,,\end{equation}
 $\Omega_\e^*$ being the convex envelope of $\Omega _\e$. 
  Then:  
\begin{itemize}
\item[(i)]  $\Gamma_\sigma ( v ^ \e)$ is locally $C ^ {1, 1}$ in $(\Omega _\e ^* ) _\sigma$; 
\smallskip
\item[(ii)]  at any $x \in (\Omega _\e ^*)_\sigma$ such that $\det   D ^ 2 (\Gamma _\sigma  ( v ^ \e)(x) \neq 0$,  it holds
$v^\e(x) = \Gamma_\sigma ( v ^ \e)(x)$; 
\smallskip
\item[(iii)]  $v ^ \e$ is a viscosity sub-solution to $- \nplap v = \overline \lambda _p (\Omega ) + \frac{p-1}{p} |\nabla v  | ^ 2$ in $\Omega _\e$. 
 \end{itemize}
\end{lemma}

\begin{proof}
We observe that by \cite[Prop.2.1.12]{CaSi} and Lemma \ref{l:approx1}(i) also $v^\e$ is semiconvex.
Statements (i) and (ii) follow now from \cite[Lemma 5]{CDDM} since, for every fixed $\e>0$, the function $v ^ \e - \log (m _ \e) $ satisfies the assumptions of such result on the convex domain $\Omega_\e^*$. 

Statement (iii) follows from part (iii) in Lemma \ref{l:approx1} above, combined with the fact that, if a smooth function $\varphi$ 
touches $v ^ \e$ from above at $x$, the smooth function $e ^ \varphi$ touches $u ^ \e$ from above at $x$. 
\end{proof}

\bigskip
{\bf Proof of Theorem \ref{t:FK}}. 
Throughout the proof we write for brevity $ \lambda _p$ in place of  $\overline \lambda _p (\Omega)$.
 Set 
\begin{equation}\label{f:defg} 
g (s): = \frac{1} {\Big ( \lambda _p + \frac{p-1}{p} s ^ 2  \Big ) ^n }\, , \qquad s \geq 0  \,,
\end{equation}
 and
 $$I _g := \int _{\R ^n} g (|\xi | ) \, d \xi \,.$$
 By direct computation in polar coordinates, the value of $I _g$ is given by
 \begin{equation}\label{f:Ig}
\begin{split} 
I _ g &  = \frac{\omega _n}{ \lambda _p ^n } \int _0 ^ { + \infty} 
  \frac{\rho ^ {n-1}} {\Big (1 + \frac{p-1}{p  \lambda _p} \rho ^ 2  \Big ) ^n } \, d \rho   =
  \frac{\omega _n}{ \lambda _p ^n } \Big (  \frac{p  \lambda _p} {p-1} \Big ) ^ {n/2} \int _0 ^ { + \infty} 
  \frac{t ^ {n-1}} {\big (1 + t^ 2  \big ) ^n } \, d t
\\
  &  = 2^{1-n} \, \pi^{(n+1)/2}   \Big (  \frac{p } {p-1} \Big ) ^ {n/2} 
\lambda _p ^{-n/2}\Gamma \Big ( \frac{n+1}{2} \Big )^{-1}\,,
\end{split} 
\end{equation}
where $\omega _n := \mathcal H ^ {n-1} (S ^ {n-1}) = 2 \pi^{n/2} \Gamma(n/2)^{-1}$.

 On the other hand, a natural idea in order to estimate $I _g$ (and hence $\lambda _p$) in terms of the measure of $\Omega$, is to apply the change of variables formula to the map $\xi = -\nabla v (x)$, with $v (x) = \log u (x)$ and
$u$ being a positive eigenfunction associated with $\lambda _p $. 

This is suggested by the fact that, as one can easily check, $v$ is a viscosity solution to 
\begin{equation}\label{f:eqv}\begin{cases}
-\nplap v = \lambda _p + \frac{p-1}{p} |\nabla v| ^ 2 & \text{ in } \Omega
\\  \noalign{\medskip} 
v = - \infty & \text{ on } \partial \Omega\,,
\end{cases}
 \end{equation}
combined with the observation that $-\nabla v$ maps $\Omega$ onto $\R ^n$, namely
 \begin{equation}\label{f:image}
 -\nabla v (\Omega) = \R ^n \,.
 \end{equation} 
Indeed, for every $p \in \R ^n$,   the minimum over $\overline \Omega$ of the function $ -v (y) - p \cdot y$  is necessarily attained a point $x$ lying in the interior  of $\Omega$ (since $v = - \infty$ on $\partial \Omega$), and at such point $x$ we have  $p =- \nabla v (x)$.  
 
 In view of \eqref{f:image}, we have 
 $$I _ g  = \int _{  -\nabla v (\Omega ) } g (|\xi | ) \, d \xi   \,,$$
 
but unfortunately the map $\xi = -\nabla v (x)$ 
  is a priori not regular enough to apply directly the area formula. Therefore, we need to proceed by approximation. 
   
Let $u ^\e$ be the supremal convolutions of $u$ according to \eqref{f:ue}, and let  $\Omega_\e$ be the domains defined in \eqref{f:omegae}. Then consider the functions $v ^ \e$ and the sets $(\Omega _\e  ^*)_\sigma$ defined  as in \eqref{f:ve} and \eqref{f:neigh}, and let $\Gamma _\sigma (v ^ \e)$ be the concave envelope of $v ^ \e$ on    $(\Omega _\e ^* ) _\sigma$.

By Lemma \ref{l:approx2} (i), we are in a position to 
apply the area formula on $(\Omega_{\e} ^* ) _\sigma $ (see \cite[Section 3.1.5]{GMS1}) to the map $\xi =-\nabla \Gamma _\sigma( v^\e)$, and we obtain 
\[
\begin{split}
\int _{ - \nabla \Gamma _\sigma( v^\e)  (  (\Omega _\e ^*)_\sigma)  } g (|\xi | ) \, d \xi 
& 
\leq \int _{  -\nabla \Gamma _\sigma( v^\e)  (  (\Omega _\e ^*)_\sigma) } g (|\xi | )\, {\rm card} ((-\nabla \Gamma _\sigma ( v ^\e) )^ {-1} (\xi) \cap (\Omega_{\e} ^* ) _\sigma )  \, d \xi  
\\
& 
= \int _{  (\Omega _\e ^*)_\sigma   } g (|\nabla \Gamma _\sigma(v ^\e) (x)| ) \, \det ( -  D ^ 2 \Gamma _\sigma( v^\e))  (x)  \, d x \,. 
\end{split}
\]
 Now, we introduce  the {\it contact set} 
 $$C_{\e, \sigma}:= \Big \{ x \in (\Omega _\e  ^*)_\sigma \ :\ v^\e(x) = \Gamma_\sigma ( v ^ \e)(x)\Big \} \,. $$  
 Thanks to Lemma \ref{l:approx2} (ii), we have
  $$ 
 \int _{  (\Omega _\e ^*)_\sigma   } g (|\nabla \Gamma _\sigma(v ^\e) (x)| ) \, \det(- D ^ 2 \Gamma _\sigma( v^\e)  )(x) \, d x  = 
 \int _{ C _{\e, \sigma}   } g (|\nabla v ^ \e(x)|)  \, \det (-D ^ 2  v ^ \e ) (x) \, d x  \,.   $$ 
Then we use the following pointwise estimates on $C_{\e, \sigma}$: 
\begin{gather} 
\det (-  D ^ 2 v^\e  ) \leq \Big (-\frac{1}{n} \Delta v ^ \e \Big ) ^ n   \label{f1}
\\
- \Delta v ^ \e  \leq - \frac {p}{(p-1) \wedge 1 } \nplap v ^ \e 
\label{f2} 
\\ 
- \nplap v ^\e \leq \lambda _p + \frac{p-1}{p} |\nabla v ^\e| ^ 2\,.   
\label{f3}
\end{gather} 
 Indeed, \eqref{f1} is consequence of the arithmetic-geometric inequality 
 observing that by construction $- D ^ 2 v_\e$ is non-negative definite on $C_{\e, \sigma}$,
 \eqref{f2} holds by  Remark \ref{r:propF} (iii), and
 \eqref{f3} holds  thanks to Lemma \ref{l:approx2}  (i) and (iii), at every point of $C _{\e, \sigma}$ where $v ^ \e$ is twice differentiable (hence a.e.\ on $C _{\e, \sigma}$). 
 
In this way we arrive at
 $$\begin{array}{ll} \displaystyle \int _{- \nabla \Gamma _\sigma( v^\e)  (  (\Omega _\e ^*)_\sigma) } g (|\xi | ) \, d \xi   & \displaystyle  \leq \int_{C _{\e, \sigma}  }  g (|\nabla v ^\e (x)|) \Big ( \frac {p}{n[(p-1) \wedge 1] } (- \nplap v  ^\e))  \Big )   ^n \, dx
 \\ \noalign{\medskip}  
  & \displaystyle\leq     \Big ( \frac {p}{n[(p-1) \wedge 1] }  \Big )  ^n  |C _{\e, \sigma} | \, , 
  \end{array}
  $$
where in the last inequality we have exploited the choice of the function $g$ in \eqref{f:defg}.   

So far, we have obtained the upper bound
$$ \int _{-\nabla \Gamma _\sigma( v^\e)  (  (\Omega _\e ^*)_\sigma)  } g (|\xi | ) \, d \xi  \leq  \Big ( \frac {p}{n[(p-1) \wedge 1] }  \Big )  ^n  |C _{\e, \sigma}| 
\,.$$ 

Now we pass to the limit in the above inequality, first as $\sigma \to 0 ^+$, and then as $\e \to 0 ^+$. 
In view of Lemma \ref{l:approx1} (iii) and \eqref{f:image}, we obtain
\[
\lim _{\e \to 0 ^+} \lim _{\sigma \to 0 ^+} \big ( -\nabla \Gamma _\sigma( v^\e)  (  (\Omega _\e ^*)_\sigma)  )  = \R ^n \qquad \text{ and } \qquad  \lim _{\e \to 0 ^+} \lim _{\sigma \to 0 ^+}  |C_{\e, \sigma}| \leq |\Omega | \,.
\]
We conclude that
\[
I _g \leq \Big ( \frac {p}{n[(p-1) \wedge 1] }  \Big )  ^n  |\Omega|  
\,.
\] 
The statement follows by inserting into the above inequality the explicit value of $I _g$ as given by \eqref{f:Ig}. 
\qed

\bigskip \bigskip

{\bf Acknowledgments.}
G.C.\ and I.F.\ 
have been supported by the Gruppo Nazionale per l'Analisi Matematica, 
la Probabilit\`a e le loro Applicazioni (GNAMPA) of the Istituto Nazionale di Alta Matematica (INdAM).

\def\cprime{$'$}

\end{document}